\documentclass[a4paper,UKenglish,cleveref,thm-restate]{lipics-v2021}

\hideLIPIcs  

\graphicspath{{./graphics/}}

\title{Beyond Windability: An FPRAS for The Six-Vertex Model} 

\author{Zhiguo Fu {}}{School of Information Science and Technology and KLAS, \and Northeast Normal University, Changchun, China 
 }{fuzg432@nenu.edu.cn}{}{}

\author{Junda Li {}}{School of Mathematics and Statistics, \and Northeast Normal University, Changchun, China}{lijd502@nenu.edu.cn}{https://orcid.org/0000-0002-9003-8066}{}

\author{Xiongxin Yang {}}{School of Information Science and Technology, \and Northeast Normal University, Changchun, China}{yangxx500@nenu.edu.cn}{https://orcid.org/0000-0002-0180-3695}{}

\authorrunning{Z.\,Fu, J.\,Li. and X.\,Yang} 

\Copyright{Zhiguo Fu, Junda Li and Xiongxin Yang} 



\category{Track A: Approximation and Online Algorithms} 

\relatedversion{} 

\acknowledgements{We want to thank anonymous reviewers for their helpful comments.}

\nolinenumbers 

\EventEditors{John Q. Open and Joan R. Access}
\EventNoEds{2}
\EventLongTitle{42nd Conference on Very Important Topics (CVIT 2016)}
\EventShortTitle{CVIT 2022}
\EventAcronym{CVIT}
\EventYear{2016}
\EventDate{December 24--27, 2016}
\EventLocation{Little Whinging, United Kingdom}
\EventLogo{}
\SeriesVolume{42}
\ArticleNo{23}

\begin{document}

\maketitle

\begin{abstract}
The six-vertex model is an important model in statistical physics and has deep connections with counting problems. There have been some fully polynomial randomized approximation schemes
(FPRAS) for the six-vertex model \cite{M.Mihail/P.Winkler/1996, J.Cai/T.Liu/P.Lu/2019}, which all require that the constraint functions are windable. In the present paper,  
we give an FPRAS for the six-vertex model with an unwindable constraint function 
by Markov Chain Monte Carlo method (MCMC). 
Different from \cite{J.Cai/T.Liu/P.Lu/2019}, we use the Glauber dynamics to design the Markov Chain depending on a circuit decomposition of the underlying graph. Moreover, we prove the rapid mixing of the Markov Chain by coupling, instead of canonical paths in \cite{J.Cai/T.Liu/P.Lu/2019}.
\end{abstract}

\section{Introduction}
\label{sec:typesetting-Introduction}

The six-vertex model originates in statistical mechanics for crystal lattices with hydrogen bonds.
A state of the model consists of an arrow on each edge such that the number of arrows pointing inwards at each vertex is exactly two. This 2-in-2-out law on the arrow configurations 
is called the ice rule (It is also called the ice-type model). Thus there are six permitted types of local configurations around a vertex, hence the name six-vertex model (See \cref{six-type}).
The six configurations 1 to 6 are associated with six possible weights $\omega_1, \omega_2, \cdots, \omega_6$.
\begin{figure}[ht]
\begin{subfigure}[t]{0.15\textwidth}
	\centering
	\includegraphics[width=1\textwidth]{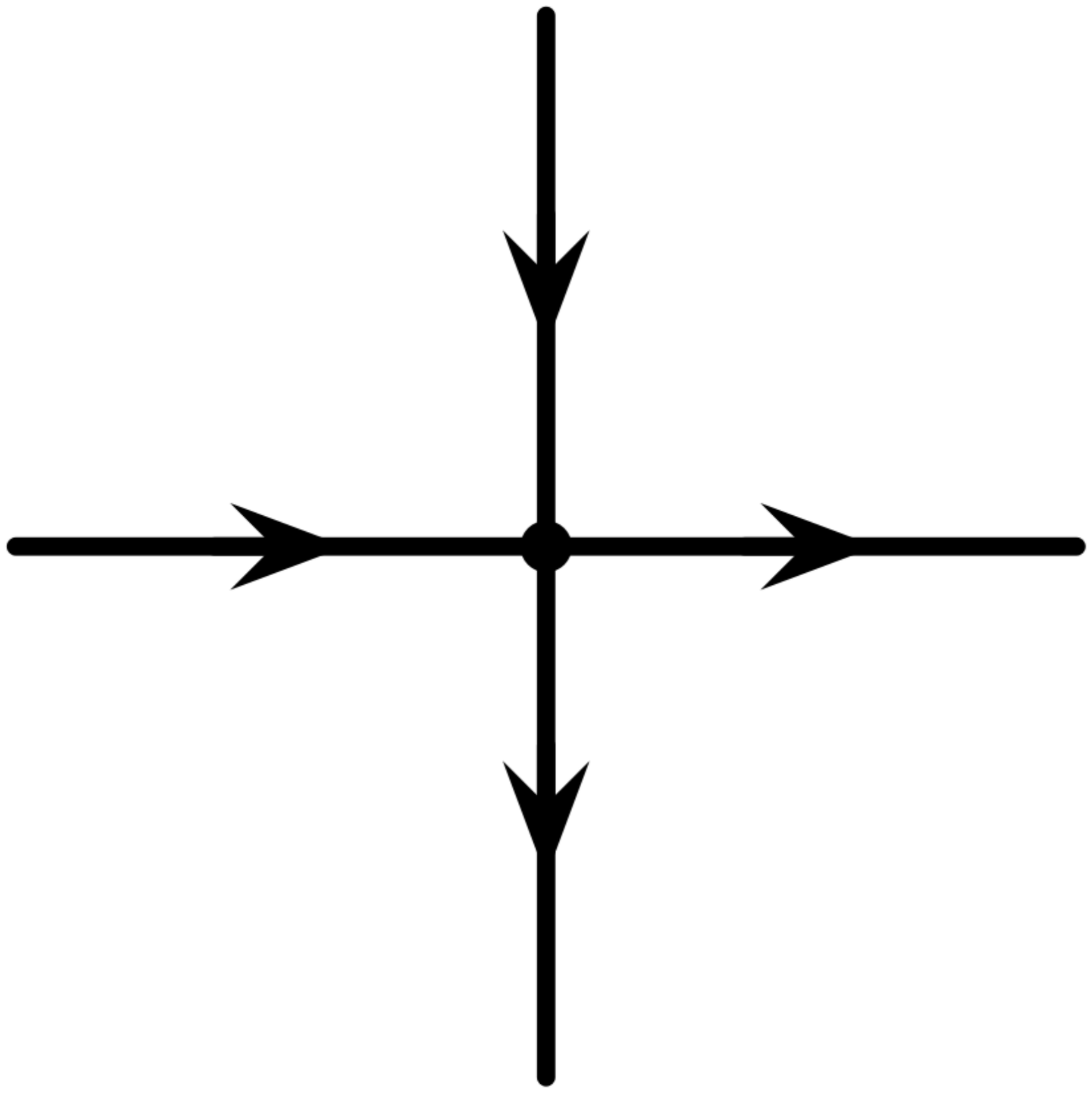}
	\subcaption*{1}
\end{subfigure}
\begin{subfigure}[t]{0.15\textwidth}
	\centering
	\includegraphics[width=1\textwidth]{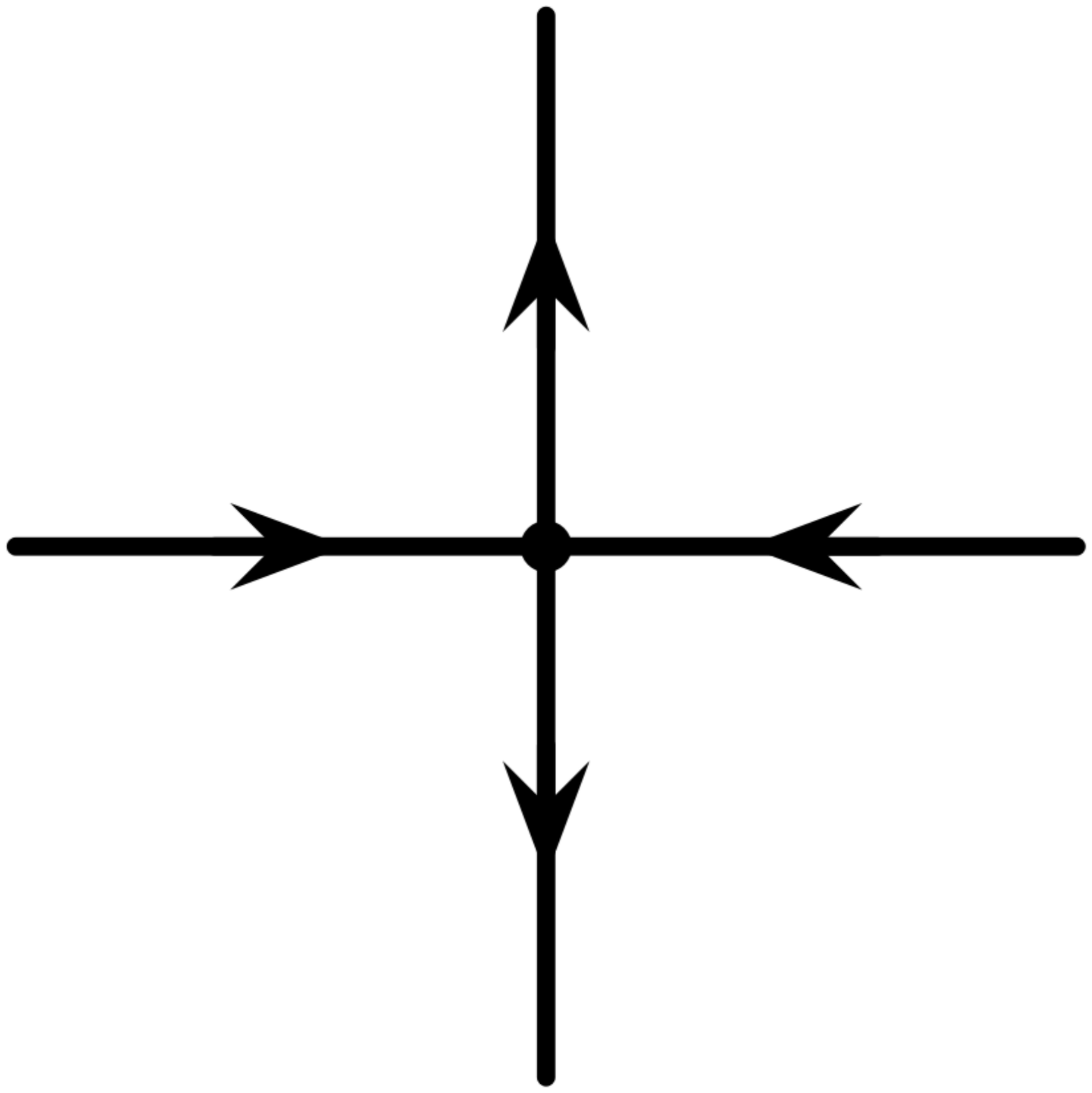}
	\subcaption*{2}
\end{subfigure}
\begin{subfigure}[t]{0.15\textwidth}
	\centering
	\includegraphics[width=1\textwidth]{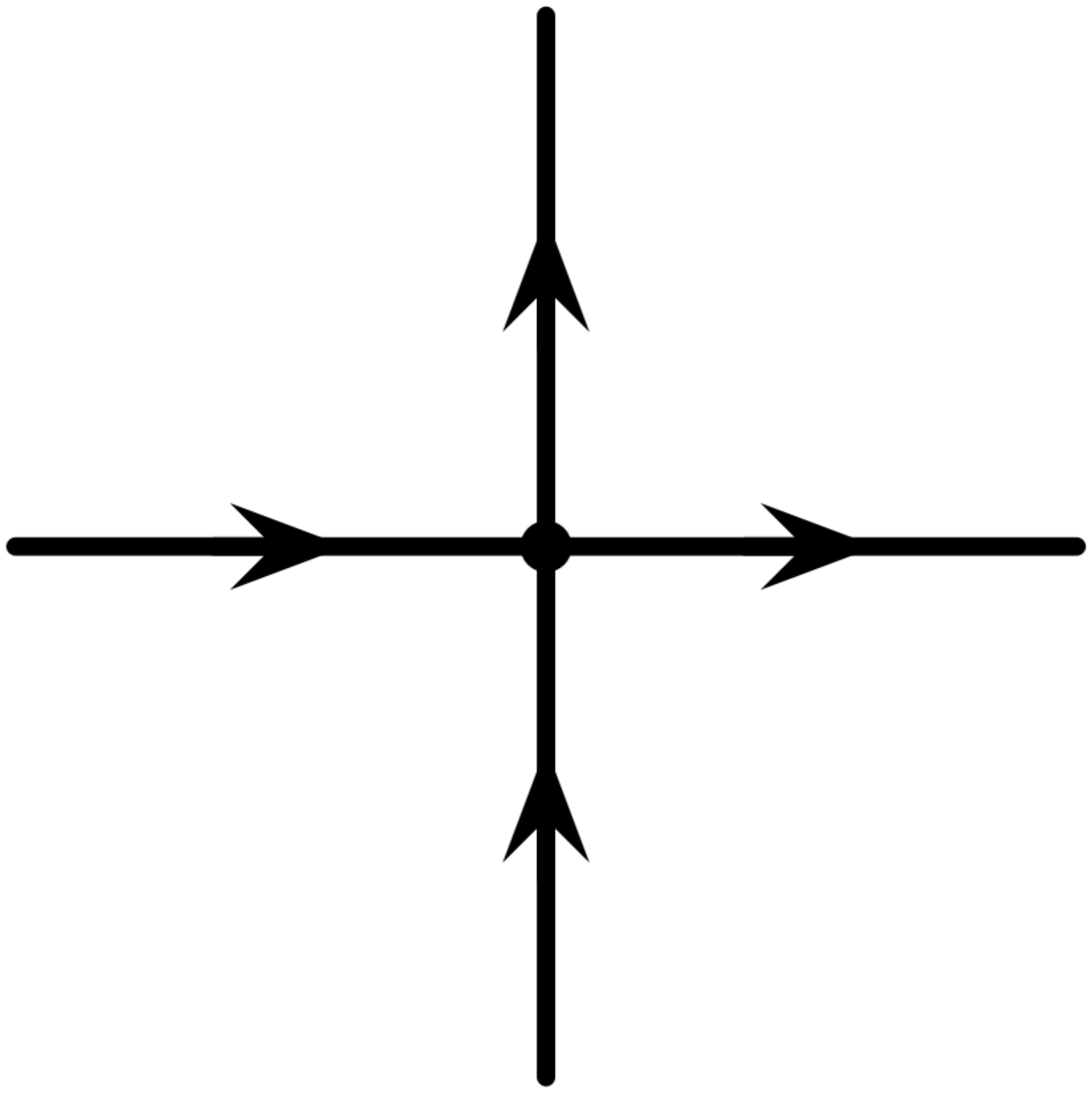}
	\subcaption*{3}
\end{subfigure}
\begin{subfigure}[t]{0.15\textwidth}
	\centering
	\includegraphics[width=1\textwidth]{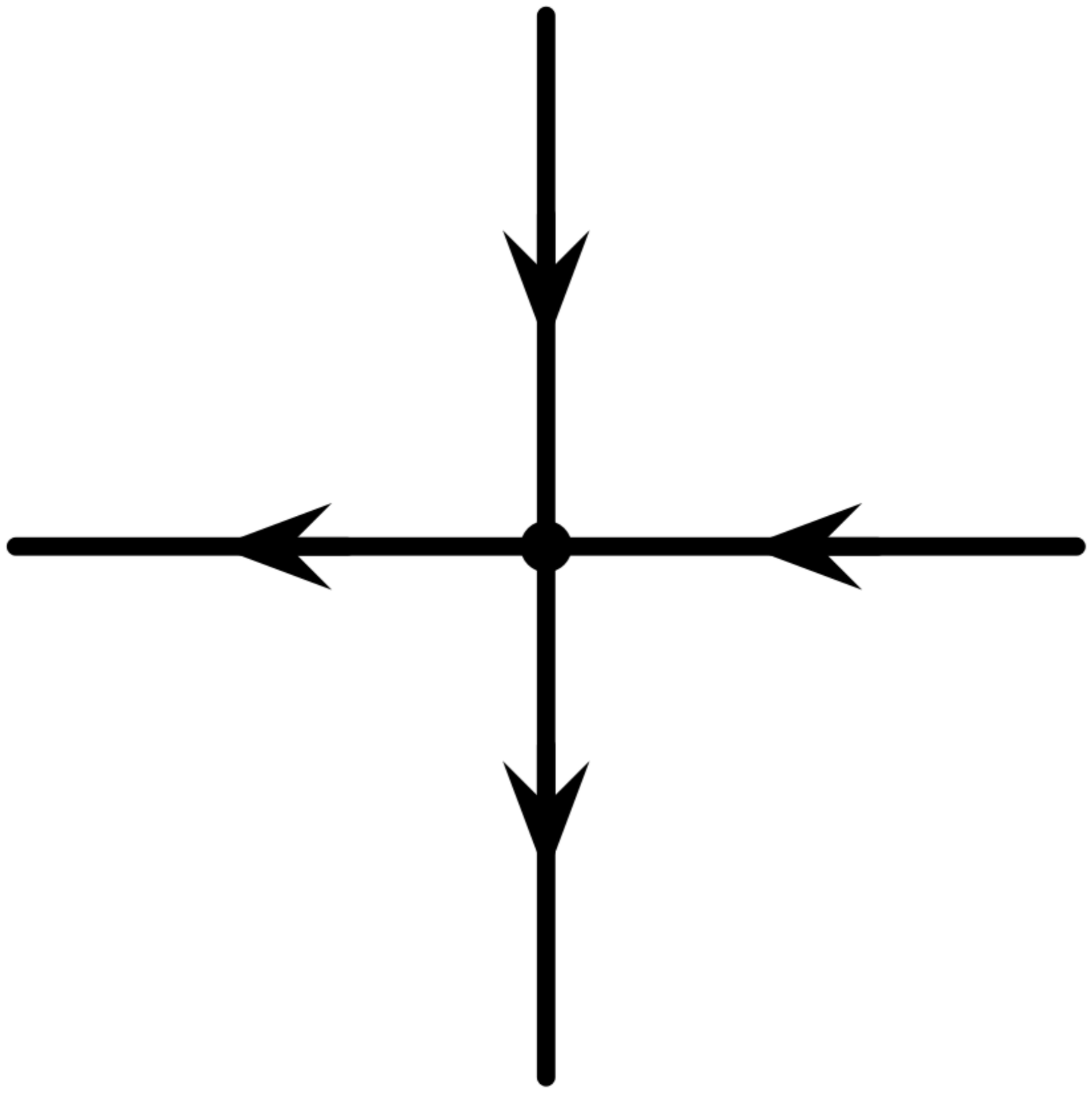}
	\subcaption*{4}
\end{subfigure}
\begin{subfigure}[t]{0.15\textwidth}
	\centering
	\includegraphics[width=1\textwidth]{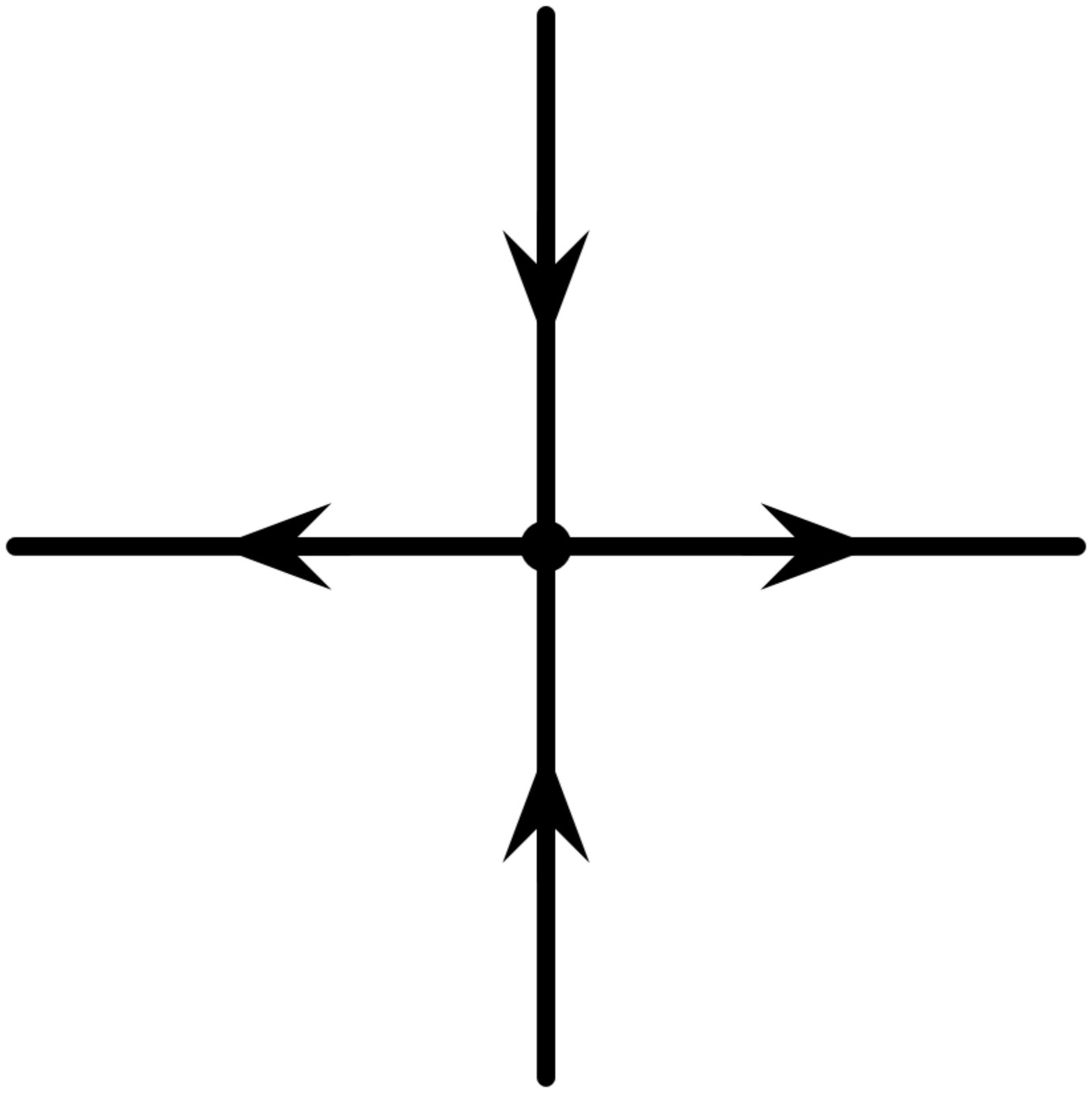}
	\subcaption*{5}
\end{subfigure}
\begin{subfigure}[t]{0.15\textwidth}
	\centering
	\includegraphics[width=1\textwidth]{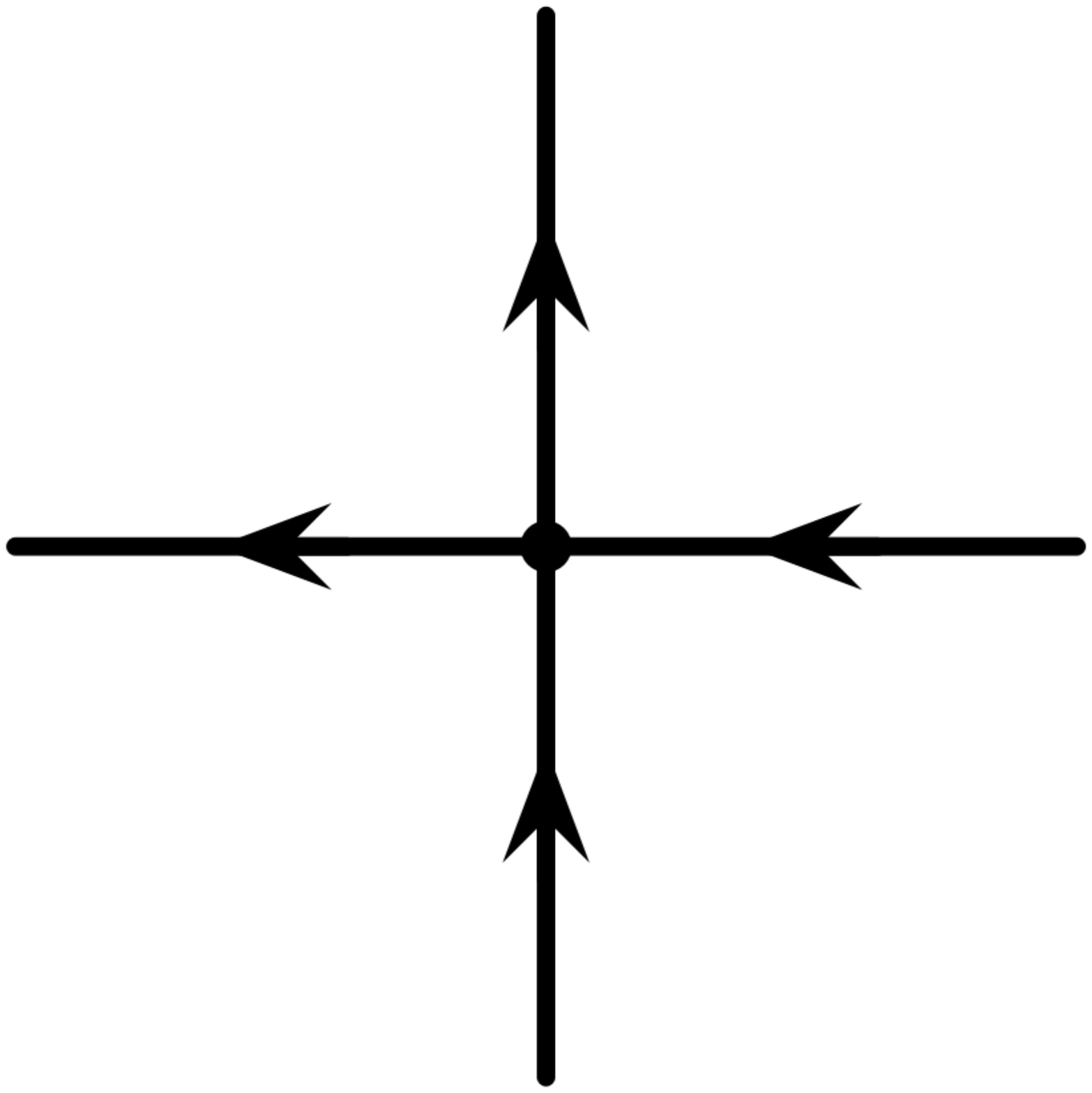}
	\subcaption*{6}
\end{subfigure}
\caption{Valid configurations of the six-vertex model}
\label{six-type}
\end{figure}

On a 4-regular graph $G$, 
the partition function of the six-vertex model is
\begin{align*}
	Z(G; \omega_1, \omega_2, \cdots, \omega_6)=\displaystyle\sum_{\sigma}\prod_{i=1}^6\omega_i^{n_i},
\end{align*} 
where $\sigma$ is the set of the orientation of $G$ such that the  incident edges of each vertex are 2-in-2-out,
and $n_i$ is the number of vertices in type $i$ ($1\leq i\leq 6$).
Since the six-vertex model was introduced by Linus Pauling in 1935 to describe the properties of ice \cite{pauling/L/1935}, 
it has attracted considerable attentions in physics, chemistry and mathematics.

Note that the partition function of the six-vertex model can be considered as a sum-of-product computation.
Thus it is a counting problem naturally. 
For example, if all the weight are 1, then the value of the partition function is the number of the Eulerian orientations of the underlying graph.
In the 1990s, Zeilberger \cite{Zeilberger/1996}, Kuperberg \cite{kuperberg/1996} showed the connection of the alternating sign matrix and the six-vertex model. 
It is also known that the six-vertex model are related to the Tutte polynomial \cite{Tutte/W.thomas/1954,Michel/1988,JA.Ellis/Criel.M/2011}. 
In general, the computational complexity of counting problems was studied in two classical frameworks:
Graph homomorphisms(GH), 
Counting Constraint Satisfication Problems(\#CSP).
Holant problems is a new framework which is expressive enough to contain GH and \#CSP as special cases \cite{Valiant/2008,J.Cai/P.Lu/M.Xia/2009}. 
A Holant instance is a graph equipped with some local constraint functions.
The six-vertex model is a Holant problem and the  constraint function  is determined by the parameters $\omega_1, \omega_2, \cdots, \omega_6$.
In particular, we note that the six-vertex model can not be expressed by GH and \#CSP.
A series theorems of complexity classifications were built for 
GH and \#CSP for exact computation \cite{GoldbergGJT/2010, BulatovDGJJR/2012, A.Bulatov/2013,
	H.Guo/T.Williams/2013,
	M.Dyer/D.Richerby/2013,
	J.Cai/X.Chen/P.Lu/2013,
	J.Cai/H.Guo/T.Williams/2016,
	S.Huang/P.Lu/2016,
	J.Cai/X.Chen/2017,
	J.Cai/Z.Fu/S.Shao/2017,
	J.Lin/H.Wang/2018,
	M.Backens/2018} and approximation complexity \cite{M.Jerrum/A.Sinclair/1989,
	A.Sinclair/1992,
	D.Randall/P.Tetali/1998,
	M.Luby/D.Randall/A.Sinclair/2001,
	L.Goldberg/R.Martin/M.Paterson/2004,
	M.Jerrum/A.Sinclair/E.Vigoda/2004,
	Goldberg/Mark/2012}.
But the results are very limited for Holant problems, in particular for approximation complexity.

The six-vertex model is an important base case to study Holant problems with asymmetric constraint functions.
And  the computational complexity of the six-vertex model was investigated in the context of Holant problems.
For the exact computation, there are the complexity classification of the six-vertex model on general graphs and planar graphs respectively \cite{J.Cai/Z.Fu/S.Shao/2017, J.Cai/Z.Fu/M.Xia/2018}. 
They proved that the six-vertex model can be divided into three categories: (1) tractable on general graphs; (2) tractable on planar graphs but 
\#P-hard on general graphs; (3) \#P-hard even on planar graphs.
For  the approximate complexity,
Mihail and Winkler  gave the FPRAS for the
unweighted case, i.e., counting the number of the Eulerian orientations of the underlying graph \cite{M.Mihail/P.Winkler/1996}.
For the weighted case, 
Cai etc.  showed that the approximate complexity of the six-vertex model is dramatically different on the two sides separated by the phase transition 
threshold from physics \cite{J.Cai/T.Liu/P.Lu/2019}. They showed that there is no FPRAS for the six-vertex model in anti-ferroelectric phase if RP$\neq$NP.
But for the remaining area, they just gave FPRAS by MCMC when the constraint function is ``windable'' and left a gap to the phase transition threshold.

Winding is proposed by McQuillan to design the canonical paths in a systematic way in \cite{C.McQuillan/2013}. 
Canonical path is an important tool to prove rapid mixing of the Markov Chain when designing MCMC. 
However, it is a difficult task to design the canonical paths that provide the routing as low congestion as possible for all state pairs of Markov Chain. %
McQuillan  reduced the task of designing canonical paths to solving a set of linear equations, which makes design of canonical paths easier. %
In details, for a Holant problem, if all the constraint functions are ``windable'', then we can find the canonical paths automatically.
But proving windable property concisely for the constraint functions of counting problems is still a difficult problem.
Therefore, Huang etc.  simplified the conditions to check if a function is windable in \cite{L.Huang/P.Lu/C.Zhang/2016}. 
From then on, the application of windability has been further extended.
The FPRAS in \cite{J.Cai/T.Liu/P.Lu/2019,J.Cai/T.Liu/P.Lu/J.Yu/2020} are all for Holant problems with windable constraint functions.

In the present paper, we give an FPRAS for the six-vertex model with unwindable constraint functions. 
We design a Markov Chain  depending on a circuit decomposition instead of the directed-loop algorithm in \cite{J.Cai/T.Liu/P.Lu/2019}, and
use path coupling to prove the rapid mixing of the Markov Chain instead of canonical paths in \cite{J.Cai/T.Liu/P.Lu/2019}. 
In details, we consider the six-vertex model with the parameters $\omega_1=1$, $\omega_3=\omega_4=b>0$ and $\omega_2=\omega_5=\omega_6=0$,
which produce an unwindable constraint function.
According to the constraint function,  the underlying graph can be decomposed  into  a set of circuits $\mathcal{C}$, and for each  circuit in $\mathcal{C}$
we can define assignments   which one-to-one corresponds to the valid configurations of the six-vertex model.
Then we define a Markov Chain on the state space consisted of the assignments of $\mathcal{C}$ by the Glauber dynamics.
We proved the Markov Chain can be mixing in
$O(n\log \frac{n}{\varepsilon })$
for $b \le \frac{1}{{\delta}}$, where both $n$ and $\delta$ are less than the number of the vertices of the underlying graph.
For some technical difficulties, we just prove the rapid mixing of the Markov Chain when the underlying graph is two-by-two-intersection free.
We believe this condition is unnecessary and will resolve it in the future.

\noindent The paper is organized as follows:
\begin{itemize}
\item In Section 2, we formalize our model and prove that its  constraint function is unwindable.
\item In Section 3, we design the Markov Chain whose state space $\Omega$ is the set of all valid configurations 
to sample
from the six-vertex model, 
which is closely related to computing the partition function \cite{M.Jerrum/L.Valiant/V.Vazirani/1986, A.Sinclair/1993}. 
\item In Section 4, we show that Markov Chain rapidly converges to its stationary distribution by path coupling, i.e., we find an efficient method to sample from the six-vertex model, and as an FPRAS to approximate the partition function.
\end{itemize}

\section{Preliminaries}
\label{sec:typesetting-Preliminaries}

\subsection{The six-vertex model}
\label{subsec:typesetting-The six-vertex model}
The six-vertex model is naturally expressed as a Holant problem, which we define as follows.
A function $f:{\left\{ {0,1} \right\}^k} \to \mathbb{C}$ is called a constraint function of arity $k$. In this paper we restrict $f$ to take non-negative values in ${\mathbb{Q}^ + }$. 
Fix a set of constraint functions $\mathcal{F}$. A function grid  $\Gamma  = \left( {G,\xi } \right)$ is a tuple, where $G = \left( {V,E} \right)$, $\xi$ labels each $v \in V$ with a function ${f_v} \in \mathcal{F}$ of arity deg($v$), and the incident edges $E\left( v \right)$ at $v$ are identified as input variables to ${f_v}$, also labeled by $ \xi $. 
Every assignment $\sigma :E \to \left\{ {0,1} \right\}$  gives an evaluation $\prod\nolimits_{v \in V} {{f_v}\left( {\sigma {|_{E\left( v \right)}}} \right)} $, where ${\sigma {|_{E\left( v \right)}}}$ denotes the restriction of $\sigma $ to ${E\left( v \right)}$.
The problem Holant($\mathcal{F}$) on an instance $\Gamma $ is to compute 
\begin{align*}
\text{Holant}\left( {\Gamma ;{\cal F}} \right) = \sum\nolimits_{\sigma :E \to \left\{ {0,1} \right\}} {\prod\nolimits_{v \in V} {{f_v}} } \left( {\sigma {|_{E\left( v \right)}}} \right).
\end{align*}
We use Holant $\left( {\mathcal{F}|\mathcal{G}} \right)$ for Holant problems over function grids with a bipartite graph $\left( {U,V,E} \right)$ where each vertex in $U$ (or $V$) is assigned a function in $\mathcal{F}$ (or $\mathcal{G}$, respectively).

To write the six-vertex model on a 4-regular graph $G=(V, E)$ as a Holant problem, consider the edge-vertex
incidence graph $G'=(U_E, U_V, E')$ of $G$. We model the orientation of an edge in $G$ by putting the
Disequality function $(\neq_2)$ (which outputs 1 on inputs 01, 10 and outputs 0 on 00, 11) on $U_E$ in $G'$.
We say an orientation on edge $e=(w, v)\in E$ is going out $w$ and into $v$ in $G$ if the edge $(u_e, u_w)$ in $G'$
takes value 1 (and ($u_e, u_v$)$\in E'$ takes the  value 0).
In the following, we consider $G'$ as the underlying graph of the six-vertex model.
We write an 4-ary function $f(x_1, x_2, x_3, x_4)$ as a  matrix $M(f)=M_{x_1x_2, x_3x_4}(f)=
\left( {\begin{array}{*{20}{c}}
		{f_{0000}}&{f_{0001}}&{f_{0010}}&{f_{0011}}\\
		{f_{0100}}&{f_{0101}}&{f_{0110}}&{f_{0111}}\\
		{f_{1000}}&{f_{1001}}&{f_{1010}}&{f_{1011}}\\
		{f_{1100}}&{f_{1101}}&{f_{1110}}&{f_{1111}}
\end{array}} \right)$.
Then the constraint function for  the six-vertex model can be expressed as a function with 
$M(f)=
\left( {\begin{array}{*{20}{c}}
		{0}&{0}&{0}&{w_1}\\
		{0}&{w_2}&{w_3}&{0}\\
		{0}&{w_4}&{w_5}&{0}\\
		{w_6}&{0}&{0}&{0}
\end{array}} \right)$.
If $w_1=w_6=a, w_3=w_4=b, w_2=\omega_5=c$, then we say the six-vertex model has arrow reversal symmetry.
In \cite{J.Cai/T.Liu/P.Lu/2019}, they proved there is no FPRAS if $a+b>c$ (anti-ferroelectric) unless RP=NP, and gave an FRPAS for $a^2+b^2<c^2$.
In the present paper, we consider the six-vertex model with the constraint function 
\begin{align}\label{f}
	M(f^*)=\left( {\begin{array}{*{20}{c}}
			0&0&0&1\\
			0&0&b&0\\
			0&b&0&0\\
			0&0&0&0
	\end{array}}\right).	
\end{align}
Note that $f^*$ has no arrow reversal symmetry. Moreover, the exact computation of the six-vertex model with $f^*$ is \#P-hard \cite{J.Cai/Z.Fu/M.Xia/2018}.

\subsection{Windable constraint function}

\begin{definition}\label{windable}
	(Windable). For any finite set J and any configuration $x \in {\left\{ {0,1} \right\}^J}$ define ${{\cal M}_x}$ to be the set of partitions of $\left\{ {i|{x_i} = 1} \right\}$ into pairs and 
	at most one singleton. A function $f:{\left\{ {0,1} \right\}^J} \to \mathbb{Q}^+$ is windable if there exist values $B\left( {x,y,M} \right) \ge 0$ for all $x,y \in {\{ 0,1\} ^J}$ and all $M \in {{\cal M}_{x \oplus y}}$ satisfying:
	\begin{itemize}
	\item $f\left( x \right)f\left( y \right) = \sum\nolimits_{M \in {{{\cal M}'}_{x \oplus y}}} {B(x,y,M)} $  for all $x,y \in {\{ 0,1\} ^J}$.
	\item $B\left( {x,y,M} \right) = B\left( {x \oplus S,y \oplus S,M} \right)$ for all $x,y \in {\{ 0,1\} ^J}$ and $S \in M \in {{\cal M}_{x \oplus y}}$.
	\end{itemize}
	Here $x \oplus S$ denotes the vector obtained by changing ${x_i}$ to $1 - {x_i}$ for the one or two elements i in S.
\end{definition}

Note that for the six-vertex model with arrow reversal symmetry, the function is windable if $a^2+b^2<c^2$, i.e., \cite{J.Cai/T.Liu/P.Lu/2019} gave an FPRAS for the six-vertex model when the function is windable. 
We will prove that the function in \cref{f} is unwindable in the following lemma, i.e., we will give an FPRAS for the six-vertex model without windability.

\begin{lemma}\label{unwindable}
	The constraint function ${f^*}$  with $b \ne 0$ in (\ref{f}) is unwindable.
\end{lemma}

\begin{proof}
	Suppose that the constraint function ${f^*}$ is windable.
	Then for $x$=0110 and $y$=1001, we have 
	\[{{\cal M}_{x \oplus y}} = \left\{ {{M_1}:\left\{ {\left( {x_1,x_2} \right),\left( {x_3,x_4} \right)} \right\},{M_2}:
		\left\{ {\left( {x_1,x_3} \right),\left( {x_2,x_4} \right)} \right\},{M_3}:\left\{ {\left( {x_1, x_4} \right),\left( {x_2,x_3} \right)} \right\}} \right\},\] 
	and there exist $B\left( {x,y,M} \right) \ge 0$ 
	such that ${f^*}(0110) \cdot {f^*}(1001) = \sum\nolimits_{M \in {{{\cal M}'}_{x \oplus y}}} {B(x,y,M)}  = {b^2}$
	by \cref{windable}.
	Moreover, we have 
	$B\left( {x,y,M} \right) = B\left( {x \oplus S,y \oplus S,M} \right)$ for $S = \left\{ {x_1, x_2} \right\}$ by \cref{windable},  i.e.,
	\begin{align*}
		B\left( {0110,1001,{M_1}} \right) = B\left( {0110 \oplus \left\{ {x_1,x_2} \right\},1001 \oplus \left\{ {x_1,x_2} \right\},{M_1}} \right) = B\left( {1010,0101,{M_1}} \right).
	\end{align*}
	Since ${f^*}(1010) \cdot {f^*}(0101) = \sum\nolimits_{M \in {{{\cal M}}_{x \oplus y}}} {B(x,y,M)}  = 0$, we have 
	\begin{align}\label{B1}
		B\left( {0110,1001,{M_1}} \right) = B\left( {1010,0101,{M_1}} \right) = 0,
	\end{align}
	Similarly, for  $S = \left\{ {x_1,x_3} \right\}$ and $S = \left\{ {x_1,x_4} \right\}$, 
	we have
	\begin{align}\label{B2}
		B\left( {0110,1001,{M_2}} \right) = B\left( {1100,0011,{M_2}} \right) = 0
	\end{align}
	and
	\begin{align}\label{B3}
		B\left( {0110,1001,{M_3}} \right) = B\left( {1111,0000,{M_3}} \right) = 0.
	\end{align}
	The  equations (\ref{B1}),(\ref{B2}) and (\ref{B3}) imply that  ${f^*}(0110) \cdot {f^*}(1001) = {b^2} = 0$. This  contradicts that $b\neq 0$.
\end{proof}

\section{Machinery}
\label{sec:typesetting-Machinery}
A valid configuration of the six-vertex model is an assignment to all edges which produces a non-zero evaluation. Let $\Omega$ be the set of all valid configurations.  
In this section, we design a Markov Chain for the six-vertex model with the state space $\Omega$. 
And then we introduce coupling to prove the rapid mixing of the Markov Chain.

\subsection{Markov chain of The six-vertex model}
\label{subsec:typesetting-Markov chain of The six-vertex model}
Dividing the four variables of the constraint function $f^*$ into two pairs $(x_1, x_3)$ and $(x_2, x_4)$,
note that the constraint function $f^*$ forces the variables in the same pair take opposite values in each valid configuration,
i.e. the corresponding edges have to be incoming and outgoing pairwisely.
Inspired by this fact, we define a circuit decomposition for the underlying
graph(See \cref{fig:label} for an example) 
\begin{itemize}
	\item Taking any vertex as the initial point, the path crosses the vertex of degree 2 directly;
	\item for the vertices of degree 4, the path travels the edges corresponding to the variables in the same pair through the vertices of degree 4, i.e., 
	if a path enters a vertex in the edge $x_1$ or $x_3$ ($x_2$ or $x_4$), then it leaves the vertex in the edge $x_3$ or $x_1$ ($x_4$ or $x_2$) respectively;
	\item  when the path travels back to the initial point, it forms a circuit. 
	Delete this circuit from the  graph and repeat the process until the graph is empty.
\end{itemize}

\begin{figure}[ht]\label{circuit}		
	\centering
	\includegraphics[scale=0.3]{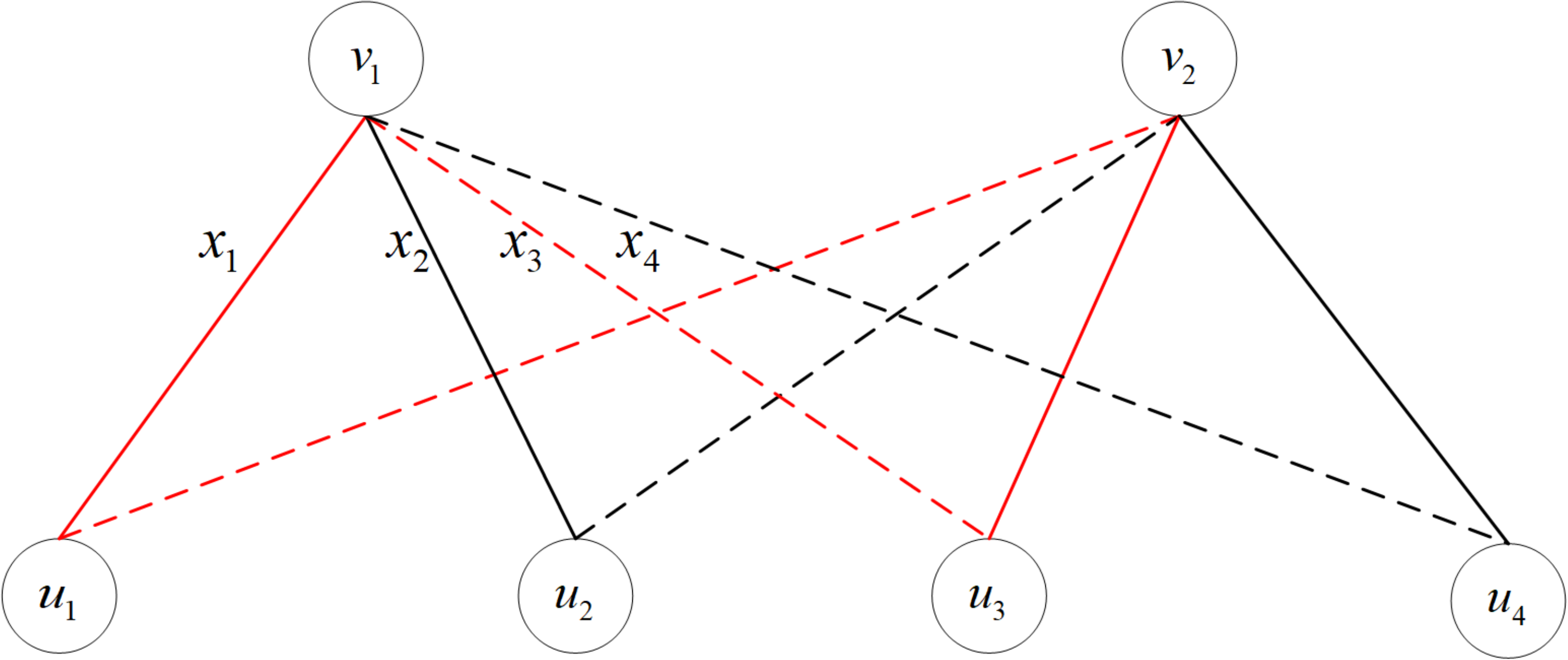}
	\caption{An example of circuit decomposition: the red edges and the black edges form two circuits 
		${C_1}$ $({v_1} \to {u_1} \to {v_2} \to {u_3} \to {v_1})$ and ${C_2}$  $({v_2} \to {u_2} \to {v_1} \to {u_4} \to {v_2})$.
		In the same circuit, the dashed edges and the solid edges take opposite assignments.
	} 
	\label{fig:label}	
\end{figure}

We denote the circuit decomposition by $\mathcal{C}$ in the following and assume that $n=|\mathcal{C}|$.
Now we redefine the valid configuration of the six-vertex model in terms of the circuit decomposition $\mathcal{C}$ before constructing the Markov Chain.
Note that the constraint function $f^*$ forces the edges in the same circuit in $\mathcal{C}$ take the values  $\{0,1\}$ alternately.
For each circuit ${C_i} $, we set an initial edge $e_i$
which corresponds to the variable $x_1$ or $x_2$ (note that each circuit in $\mathcal{C}$ travels at least one edge corresponding to the variable $x_1$ or $x_2$). 
The assignments of all the edges in $C_i$ are determined by the assignment of $e_i$ in a valid configuration.
Thus we can define the assignment of the circuit $C_i$ as the value of the initial edge $e_i$, i.e., a valid configuration assigns a value for each circuit.
Conversely, if we assign a value for each $C_i\in\mathcal{C}$, then it gives a configuration for the six-vertex model.
In the following, we set $\sigma$ is a 0-1 string $(C_1\ C_2\ \dots \ C_{|\mathcal{C}|})$ where each $C_i\in\mathcal{C}$ takes the value $0$ or $1$ and denote it by $\sigma (C_i)=0$ or $1$. 
We say that $\sigma$ is a configuration since it assigns a value for each edge in fact. 
If two circuits ${C_i}$ and ${C_j}$  share a common vertex, then we say that ${C_i}$ is a neighbor of  ${C_j}$.
Note that $f^*(1100)=0$. Thus for two adjacent circuits $C_i$, $C_j$, there is at most one which takes the value 1 in a valid configuration. 
We denote the set of neighbors of ${C_i}$ by $\Gamma ({C_i})$ and set
\begin{align*} 
\delta _{C_i}=|\Gamma(C_i)| \ \rm{and} \  \delta=\displaystyle\max_{C_i\in\mathcal{C}}\delta_{C_i}.
\end{align*} 

Now we define the Markov Chain $\{Z_t\}$ on the state space $\Omega $ by the similar idea with \cite{Luby/Michael/Vigoda/1999}. For any configuration $\sigma\in\Omega$,
the transition probability matrix $P$ of $\{Z_t\}$ is as follows:
\begin{itemize}
	\item Choose an arbitrary circuit ${C_i} \in \mathcal{C}$ at random.
	\item Let 
	\begin{align*}
		\sigma ' = \left\{ {\begin{array}{*{20}{c}}
				{{\rm{only \ move}}\;{C_i}\;{\rm{to}}\;1\;{\rm{in}}\;\sigma \  {\rm{ with \ the \  probability}} \ \frac{{{b^{{\delta_{{C_i}}}}}}}{{1 + {b^{{\delta _{{C_i}}}}}}}}\\
				{{\rm{only \ move}}\;{C_i}\;{\rm{to}}\;0\;{\rm{in}}\;\sigma \  {\rm{ with \ the \ probability}} \ \frac{1}{{1 + {b^{{\delta _{{C_i}}}}}}}}
		\end{array}} \right.
	\end{align*}
	\item If $\sigma '$ is a valid configuration, move to state $\sigma'$, otherwise remain at state $\sigma $.
\end{itemize}
\begin{proposition}\label{P}
	The transition probability matrix P of $\{Z_t\}$ has the following properties: 
	\begin{itemize}
	\item aperiodicity: gcd$\left\{ {t:P_{ii}^t > 0} \right\} = 1$ for all $\sigma_i \in \Omega $;
	\item irreducibility: there exists  $t$ such that there is a positive probability of going from state $\sigma_i$ to state $\sigma_j$ after $t$ steps, $P_{ij}^t > 0$, for all $\sigma_i, \sigma_j \in \Omega $.
	\end{itemize}
\end{proposition}

\begin{proof}
	For each state $\sigma_i\in\Omega$, there is a positive probability such that it stays the current state. Thus we have $gcd\left\{ {t:{P_{ii}} > 0} \right\} = 1$
	and aperiodicity is proved. 
	
	For each state $\sigma\in\Omega$, by moving the circuits which take the value 1 in $\sigma$ to 0 step by step, 
	there is a path to the state $\sigma_0$ in which all the circuits take the value 0  and vice versa, i.e., 
	for any two states, there is at least one path connected them by crossing $\sigma_0$. Irreducibility is proved.
\end{proof}

Based on the aperiodicity and irreducibility, we know that the chain $\{{Z_t}\}$ has a unique limiting distribution which is called the stationary distribution $\pi $, i.e.
\begin{align*}
	\mathop {\lim }\limits_{t \to \infty } P_{ij}^t = {\pi _j}, \ \text{for all} \ \sigma_i,\sigma_j \in \Omega.	
\end{align*}

The distribution of the partition function $\mu$ is as follows
\begin{align*}
	\mu (\sigma ) = \frac{{\prod\limits_{{C_i}:\sigma ({C_i}) = 1} {{b^{{\delta _{{C_i}}}}}} }}{{\sum\limits_{\sigma  \in \Omega } {\prod\limits_{{C_i}:\sigma ({C_i}) = 1} {{b^{{\delta _{{C_i}}}}}} } }}.
\end{align*}
Note that ${\sum\limits_{\sigma  \in \Omega } {\prod\limits_{{C_i}:\sigma ({C_i}) = 1} {{b^{{\delta _{{C_i}}}}}} } }$
is the partition function of the six-vertex model with the constraint function $f^*$.
Moreover, we emphasize that  the distribution $\mu$ is the unique stationary distribution of $\{Z_t\}$, because $\mu$ satisfies the detailed balance conditions 
\[{\mu _i}P_{ij}^t = {\mu _j}P_{ji}^t\] for all $\sigma_i, \sigma_j \in \Omega$.
Our goal is to bound the time when the chain is close to the stationary distribution, i.e., we are interested in the mixing time $\tau $:
\begin{align*}
	\tau (\varepsilon ) = \mathop {\max }\limits_i \min \{ t:{\Delta _i}(t') \le \varepsilon \ \text{for all} \ t' \ge t\},	
\end{align*}
where ${\Delta _i}(t)$ is the total variation distance between $P_i^t$ and $\pi$, i.e., 
\begin{align*}
	{\Delta _i}(t) = \frac{1}{2}\sum\limits_{\sigma_j \in \Omega } {\left| {P_{ij}^t - {\pi _j}} \right|}.
\end{align*}

\subsection{Coupling}
\label{subsec:typesetting-Coupling}

We use coupling to prove the rapid mixing of $\{Z_t\}$.
Coupling constructs a stochastic process $\left( {{X_t},{Y_t}} \right)$ such that 
$\{X_t\}$ and  $\{Y_t\}$ are copies of the $\{Z_t\}$ and if ${X_t} = {Y_t}$, then ${X_{t + 1}} = {Y_{t + 1}}$. To reduce the difficulty of constructing the coupling, we use path coupling in \cite{R.Bubley/M.Dyer/1997} as a tool in the following analysis. 

Firstly, we define the notion of neighbors and paths in $\Omega $. 
For a pair of states $\sigma $, ${\sigma _{{C_i}}} \in \Omega $, if $\sigma(C_i)=0$, 
$\sigma_{C_i}(C_i)=1$ and $\sigma(C_j)=\sigma_{C_i}(C_j)$ for $j\neq i$, then $\sigma $ and ${\sigma _{{C_i}}}$ are neighbors and denote it by $\sigma  \sim {\sigma _{{C_i}}}$. 
We call $\tau  = \left( {{\tau _0}, {\tau _1}, \cdots ,{\tau _k}} \right)$ a simple path if all ${\tau _i}$ are distinct and ${\tau _0} \sim {\tau _1} \sim \cdots  \sim {\tau _k}$. And for each $\sigma, \eta \in \Omega$, we denote the set of the simple paths connected them by
\begin{align*}
	\rho(\sigma ,\eta ) = \{ \tau :\sigma  = {\tau _0}, \ \eta  = {\tau _k},\tau \ \text{is a simple path} \}. 
\end{align*}
We assume that $X_t = \sigma$ and $Y_t = \sigma _{{C_i}}$, and denoted the next state by $X_{t+1}=\sigma '$ and  $Y_{t+1}={\sigma '_{{C_i}}}$ respectively. 
Note that both Markov Chains $\{X_t\}$, $\{Y_t\}$ use the same random walk at each step in path coupling, even if the walk cannot produce a valid state (if the produced state is not valid, the chain will stay the current state). 
The following lemma about path coupling follows from \cite{M.Dyer/C.Greenhill/2000MC}. 
\begin{lemma}\label{coupling}
	Let $\Phi $ be an integer-valued metric defined on $\Omega  \times \Omega $ which takes values in $\left\{ {0, \ldots ,D} \right\}$ and for all $\sigma ,\eta  \in \Omega $, then
	there exists $\tau\in \rho(\sigma, \eta)$ such that
	\begin{align*}
		\Phi \left( {\sigma ,\eta } \right) = \sum\limits_i {\Phi \left( {{\tau _i},{\tau _{i + 1}}} \right)}. 
	\end{align*}
	Suppose there exists $\beta  < 1$ and a coupling of the Markov Chain such that for all $\sigma ,{\sigma _{{C_i}}} \in \Omega $:
	\begin{align*}
		E\left[ {\Phi \left( {\sigma ',{{\sigma '}_{{C_i}}}} \right)} \right] \le \beta  \cdot \Phi \left( {\sigma ,{\sigma _{{C_i}}}} \right).
	\end{align*}
	Then the mixing time is 
	\begin{align*}
		\tau \left( \varepsilon  \right) \le \frac{{\log \left( {D{\varepsilon ^{ - 1}}} \right)}}{{1 - \beta }}.
	\end{align*}
\end{lemma}

\section{Analysis}
\label{sec:typesetting-Analysis}
For the circuit decomposition $\mathcal{C}$ of the underlying graph of the six-vertex model, if there are no three circuits intersecting with one another, then we say the underlying graph is 
two-by-two-intersection free.
In this section, we analyse the mixing speed of $\{Z_t\}$ on two-by-two-intersection free graphs.

\subsection{Potential Function}
\label{subsec:typesetting-Potential Function}
We will define the potential function as the distance of states associated with the path coupling at step $t$, such that if the potential function is zero, then the two chains  will eventually reach the same state. 
According to path coupling, we consider a pair of states $\sigma \sim {\sigma _{{C_i}}}$. 
If the next move is on the ``bad neighbors'' of $C_i$, then the move will not work in one of chains and  causes more differences between two chains.
Thus we will define a potential function associated with the number of the ``bad neighbors'' of $C_i$.

For a pair of adjacent state $\sigma \sim \sigma_{{C_i}}$ of $\{X_t\}$ and $\{Y_t\}$, the circuit  ${C_j}$ is blocked if ${C_j}$ has at least one neighbor which take the value 1 in both $\sigma $ and ${\sigma _{{C_i}}}$.
Denote the set of blocked neighbors of ${C_i}$ by 
\begin{align*}
	B(\sigma ,{C_i}) = \left\{ {{C_j}:{C_j} \in \Gamma ({C_i}), \ \text{there exists} \ {C_k} \in \Gamma \left( {{C_j}} \right) \ \text{s.t.} \ \sigma \left( {{C_k}} \right) = 1} \right\}.
\end{align*}
And the set of unblocked circuits is $\Gamma \left( {{C_i}} \right)\backslash B\left( {\sigma ,{C_i}} \right)$ for $\sigma$ and $\sigma_{{C_i}}$. For the current state $(\sigma, \sigma_{C_i})$ of the coupling $(X_t, Y_t)$, the move on an unblocked circuit only works in $\sigma$, and it makes more difference between $\{X_t\}$ and $\{Y_t\}$.
The move on a blocked circuit will keep the difference between the states of $\{X_t\}$ and $\{Y_t\}$.
Note that unblocked circuits are ``bad neighbors'' we mentioned above.
For some technical reasons, we define the potential function as the number of the neighbors of ${C_i}$ minus a positive constant $w<1$ times the number of the blocked neighbors of ${C_i}$, instead of the number of the unblocked neighbors of $C_i$.
Specifically, for $w = \frac{{{b\delta }}}{{{b^\delta } + 2}}$ (note that $w<1$ if $b\le \frac{1}{\delta}$), the potential function $\Phi $ is as follows,
\begin{align}\label{potential}
	\Phi (\sigma ,{\sigma _{{C_i}}}) = {\delta _{{C_i}}} - w\left| {B(\sigma ,{C_i})} \right|.
\end{align}
For any states $\sigma $ and $\eta $, the potential function is 
\begin{align*}
	\Phi (\sigma ,\eta ) = \mathop {\min }\limits_{\tau  \in \rho(\sigma ,\eta )} \sum\limits_i {\Phi ({\tau _i},{\tau _{i + 1}})}.
\end{align*}
This potential function $\Phi $ clearly satisfies the following conditions for any $\sigma$, $\eta  \in \Omega$:
\begin{itemize}
\item $\Phi (\sigma ,\eta ) \le \Phi (\sigma ,\varsigma ) + \Phi (\varsigma ,\eta )$, for any $\varsigma  \in \Omega $;
\item $\Phi (\sigma ,\eta ) = \Phi (\eta ,\sigma )$;
\item $\Phi (\sigma ,\eta ) \ge 0$;
\item $\Phi (\sigma ,\eta ) = 0 \Leftrightarrow \sigma  = \eta$.
\end{itemize}

\subsection{Analysis on the Potential Function}
\label{subsec:typesetting-Analysis on the Potential Function}

Recall that we set $X_{t}=\sigma$ and $Y_{t}={\sigma _{{C_i}}}$, where $\sigma $ and ${\sigma _{{C_i}}}$ are neighbors, $\sigma_{C_i}(C_i)=1$ and $\sigma(C_j)=\sigma_{C_i}(C_j)$ for $j\neq i$, 
and the next states are $X_{t+1}=\sigma '$ and  $Y_{t+1}={\sigma '_{{C_i}}}$. Let $\Phi  = \Phi \left( {\sigma ,{\sigma _{{C_i}}}} \right)$ and $E\left[ {\Delta \Phi } \right] = E\left[ {\Phi \left( {\sigma ',{{\sigma '}_{{C_i}}}} \right) - \Phi \left( {\sigma ,{\sigma _{{C_i}}}} \right)} \right]$.  
We now analyze $E\left[ {\Delta \Phi } \right]$ to find the constant $\beta$ in \cref{coupling}.

We classify the circuits in $\mathcal{C}$ into four types: ${C_i}$, ${C_j}$ as a neighbor of ${C_i}$, ${C_k}$ as a neighbor of the neighbors of ${C_i}$, 
and $C_{\ell}$ as others. In the following analysis, we will only discuss the effect of moves on ${C_i}$, ${C_j}$ and ${C_k}$, because the moves on $C_{\ell}$ have no effect on $\Phi $. For any circuit ${C_x}$, let
\begin{align*}
	E\left[ {{\Delta ^{ + {C_x}}}\Phi } \right] = E\left[ {\Delta \Phi |\text{Markov Chain attempts to move $C_x$ to $1$}}\right],
\end{align*}
\begin{align*}
	E\left[ {{\Delta ^{ - {C_x}}}\Phi } \right] = E\left[ {\Delta \Phi |\text{Markov Chain attempts to move $C_x$ to $0$}}\right],	
\end{align*}
\begin{align*}
	E\left[ {{\Delta ^{{C_x}}}\Phi } \right] = E\left[ {\Delta \Phi |\text{Markov Chain attempts to move $C_x$ }}\right],	
\end{align*}
then 
\begin{align*}
	E\left[ {{\Delta ^{{C_x}}}\Phi } \right] = \frac{{{b^{{\delta _{{C_x}}}}}}}{{1 + {b^{{\delta _{{C_x}}}}}}}E\left[ {{\Delta ^{ + {C_x}}}\Phi } \right] + \frac{1}{{1 + {b^{{\delta _{{C_x}}}}}}}E\left[ {{\Delta ^{ - {C_x}}}\Phi }\right].
\end{align*}\par
By the above notations, we can rewrite $E\left[ {\Delta \Phi } \right]$ as follows:
\begin{align}\label{eq8}
	E\left[ {\Delta \Phi } \right] = \frac{1}{n}\left[ {E\left[ {{\Delta ^{{C_i}}}\Phi } \right] + \sum\limits_{{C_j} \in \Gamma ({C_i})} {E\left[ {{\Delta ^{{C_j}}}\Phi } \right]}  + \sum\limits_{{C_k} \in \Gamma (\Gamma ({C_i}))} {E\left[ {{\Delta ^{{C_k}}}\Phi } \right]} } \right].
\end{align}\par
Here, $\Gamma (\Gamma ({C_i}))$ is the set of the neighbors of the neighbors of $C_i$. 

Now we consider the moves on ${C_i}$, ${C_j}$ or ${C_k}$ separately. 
\begin{itemize}
	\item {Move on ${C_i}$:\par In this case, a move on ${C_i}$ will work in both $\{X_t\}$ and $\{Y_t\}$. 
		After the move, $\left\{ {{X_t}} \right\}$ and  $\left\{ {{Y_t}} \right\}$ will reach the same state which is $\sigma $ or ${\sigma _{{C_i}}}$. 
		Thus, we have 
		\begin{align*}
			E\left[ {{\Delta ^{ + {C_i}}}\Phi } \right] = E\left[ {{\Delta ^{ - {C_i}}}\Phi } \right]=  - {\delta _{{C_i}}} + w\left| {B\left( {\sigma ,{C_i}} \right)}\right|, 
		\end{align*}\par 
		i.e.,  
		\begin{align}E\left[ {{\Delta ^{{C_i}}}\Phi } \right] = - {\delta _{{C_i}}} + w\left| {B(\sigma ,{C_i})} \right|.\end{align}}
	\item Move on ${C_j}$:\par
	Since ${C_j}$ is a neighbor of ${C_i}$, ${C_j}$ must take the  value $0$ in $\sigma_{C_i}$. Note that $\sigma$ and $\sigma_{C_i}$ 
	differ only at $C_i$, so ${C_j}$ takes the value $0$ in $\sigma$. Thus, $E[{\Delta ^{ - {C_j}}}\Phi ] = 0$.
	
	Consider the move which attempts to move ${C_j}$ to $1$. If ${C_j}$ is unblocked,  then this move will only work in $\{X_t\}$, and we have
	\begin{align*}
			\Delta ^{+C_j} \Phi  
			&= \Phi (\sigma ',{{\sigma '}_{{C_i}}}) - \Phi (\sigma ,{\sigma _{{C_i}}})\\ 
			&= \Phi ({\sigma _{{C_j}}},{\sigma _{{C_i}}}) - \Phi (\sigma ,{\sigma _{{C_i}}})\\
			&\le \Phi ({\sigma _{{C_j}}},\sigma ) + \Phi (\sigma ,{\sigma _{{C_i}}}) - \Phi (\sigma ,{\sigma _{{C_i}}})\\
			&= \Phi ({\sigma _{{C_j}}},\sigma )\\
			&= \delta _{{C_j}} - w\left| {B(\sigma ,{C_j})} \right|.
	\end{align*}
	If ${C_j}$ is blocked, then this move will work neither in $\{X_t\}$ nor in $\{Y_t\}$. And both chains will stay the current state. Thus $\Delta ^{+{C_j}}\Phi = 0$.
	Thus we have 
	\begin{align}\label{ECj}
		E\left[ {{\Delta ^{{C_j}}}\Phi } \right] \le \left\{ {\begin{array}{*{20}{l}}
				{\frac{{{b^{{\delta _{{C_j}}}}}}}{{1 + {b^{{\delta _{{C_j}}}}}}} \cdot \left( {{\delta _{{C_j}}} - w\left| {B(\sigma ,{C_j})} \right|} \right)}&{\text{if} \ {C_j} \notin B\left( {\sigma ,{C_i}} \right)},\\
				0&{\text{otherwise}}.
		\end{array}} \right.
	\end{align}
	\item Move on ${C_k}$:\par
	Firstly, we assume that ${C_k}$ takes the value $1$ in  $\sigma $, then ${C_k}$  also takes the value $1$ in ${\sigma _{{C_i}}}$. The move which attempts to move ${C_k}$ to $1$ will work in both $\{X_t\}$ and $\{Y_t\}$. And both of them will stay the current state. Thus
	\begin{align}\label{+Ck=1}
		{\Delta ^{ + {C_k}}}\Phi = 0.
	\end{align}
	Then we consider the move which attempts to move ${C_k}$ to $0$. This move will work in both $\{X_t\}$ and $\{Y_t\}$ and it might change some ${C_j}$ from blocked to unblocked. The set of such ${C_j}$ is defined as follows: 
	\begin{align*}
		{\alpha _{{C_k}}} = \left\{ {{C_j}:{C_j} \in B\left( {\sigma ,{C_i}} \right),\;  \sigma \left( {{{C'}_k}} \right) = 0} \ \text{for any} \ {{C'}_k} \in \Gamma \left( {{C_j}} \right) \right\}.
	\end{align*}
	We have
	\begin{align*}
			{\Delta ^{ - {C_k}}}\Phi  
			&= \Phi (\sigma ',{\sigma '_{{C_i}}}) - \Phi (\sigma ,{\sigma _{{C_i}}}) \\
			&= {\delta _{{C_i}}} - w\left| {B\left( {\left( {{\sigma _{{C_i}}}} \right)_{{C_k}}^ - ,{C_i}} \right)} \right| - \left( {{\delta _{{C_i}}} - w\left| {B(\sigma ,{C_i})} \right|} \right) \\
			&= w\left( {\left| {B(\sigma ,{C_i})} \right| - \left| {B\left( {\left( {{\sigma _{{C_i}}}} \right)_{{C_k}}^ - ,{C_i}} \right)} \right|} \right) \\
			&= w\left| {{\alpha _{{C_k}}}} \right|,
	\end{align*}
where ${\left( {{\sigma _{{C_i}}}} \right)_{{C_k}}^ - }$ denotes the configuration in which all the circuits take the same value as
$\sigma_{C_i}$ except for ${\left( {{\sigma _{{C_i}}}} \right)_{{C_k}}^ - }(C_k)=0$ and $\sigma_{{C_i}}(C_k)=1$.
Thus
	\begin{align}\label{-Ck=1}
		{\Delta ^{ - {C_k}}}\Phi  = \left\{ {\begin{array}{*{20}{l}}
				{w\left| {{\alpha _{{C_k}}}} \right|}&{\text{if} \ \sigma \left( {{C_k}} \right) = 1},\\
				0&{\text{otherwise}.}
		\end{array}} \right.
	\end{align}
	Secondly, we assume that ${C_k}$ takes the value $0$ in $\sigma $. The move which attempts to move ${C_k}$ to $0$ will work in both $\{X_t\}$ and $\{Y_t\}$. And both will stay the current state. Thus
	\begin{align}\label{-Ck=0}
		{\Delta ^{ - {C_k}}}\Phi = 0. 
	\end{align}
	Then consider the move which attempts to move ${C_k}$ to $1$. Recall that the graph is two-by-two-intersection free, so ${C_k}$ and ${C_i}$ are not adjacent.
	If there exists a neighbor of ${C_k}$ which takes the value $1$ in $\sigma $, then this move will not work in neither of two chains. Thus both chains stay the current states and $\Delta ^{+C_k}\Phi  = 0$; if all the neighbors of $C_k$ take the value 0 in $\sigma$ and denote the fact by $\sigma \left( {\Gamma \left( {{C_k}} \right)} \right) = 0$, then this move will work in both chains. Note that moving $C_k$ to 1 which works in both chains effects the potential function by changing ${C_j}$  from unblocked to blocked. The set of such ${C_j}$ is defined as follows:
	\begin{align*}
		{\beta _{{C_k}}} = \left\{ {{C_j}:{C_j} \notin B\left( {\sigma ,{C_i}} \right),{C_j} \in \Gamma \left( {{C_i}} \right) \cap \Gamma \left( {{C_k}} \right)} \right\}.
	\end{align*}
	We have
	\begin{align*}
			{\Delta ^{+C_k}}\Phi  
			&= \Phi (\sigma ',{{\sigma '}_{{C_i}}}) - \Phi (\sigma ,{\sigma _{{C_i}}}) \\
			&= w\left( {\left| {B(\sigma ,{C_i})} \right| - \left| {B(\sigma_{C_k},{C_i})} \right|} \right) \\
			&=  - w\left| {{\beta _{{C_k}}}} \right|.
	\end{align*}
	Thus 
	\begin{align}\label{+Ck=0}
		{\Delta ^{+C_k}}\Phi  = \left\{ {\begin{array}{*{20}{l}}
				{ - w\left| {{\beta _{C_k}}} \right|}&{\text{if} \ \sigma \left( {\Gamma \left( {{C_k}} \right)} \right) = 0 \ {\text{and}} \ \sigma({C_k}) = 0},\\
				0&{\text{otherwise}}.
		\end{array}} \right.
	\end{align}
	Combining \cref{-Ck=1} and \cref{-Ck=0}, we have
	\begin{align}\label{-Ck}
		{\Delta ^{ - {C_k}}}\Phi  = \left\{ {\begin{array}{*{20}{l}}
				0&{\sigma({C_k}) = 0 }, \\
				{w\left| {{\alpha _{{C_k}}}} \right|}&{\sigma({C_k}) = 1 }.
		\end{array}} \right.
	\end{align}
	Combining \cref{+Ck=1} and \cref{+Ck=0}, we have
	\begin{align}\label{+Ck}
		{\Delta ^{{\rm{ + }}{C_k}}}\Phi  = \left\{ {\begin{array}{*{20}{l}}
				0&{{\rm{if}}\;\sigma \left( {{C_k}} \right) = 1},\\
				{ - w\left| {{\beta _{{C_k}}}} \right|}&{{\rm{if}}\;\sigma \left( {\Gamma \left( {{C_k}} \right)} \right) = 0\;{\rm{and}}\;\sigma \left( {{C_k}} \right) = 0},\\
				0&{{\rm{otherwise}}}.
		\end{array}} \right.
	\end{align}
	Combining \cref{-Ck} and \cref{+Ck}, for any neighbor of neighbors of $C_i$, i.e., ${C_k}$, we have
	\begin{align}\label{ECk}
		E\left[ {{\Delta ^{{C_k}}}\Phi } \right] = \left\{ {\begin{array}{*{20}{l}}
				{\frac{{\left| {{\alpha _{{C_k}}}} \right|}}{{1 + {b^{{\delta _{{C_k}}}}}}}w}&{{\rm{if}}\;\sigma \left( {{C_k}} \right) = 1},\\
				{ - \frac{{{b^{{\delta _{{C_k}}}}}\left| {{\beta _{{C_k}}}} \right|}}{{1 + {b^{{\delta _{{C_k}}}}}}}w}&{{\rm{if}}\;\sigma \left( {\Gamma \left( {{C_k}} \right)} \right) = 0\;{\rm{and}}\;\sigma \left( {{C_k}} \right) = 0},\\
				0&{{\rm{otherwise}}}.
		\end{array}} \right.
	\end{align}
	
\end{itemize}

After analyzing the moves on $C_i$, $C_j$ and $C_k$ separately, we are ready to bound $E[\Delta \Phi ]$. It seemed that we can compute $E[\Delta \Phi ]$ by \cref{eq8}. But it can not work in practice, since we do not know the number of $C_k$. Thus, we can not give a sharp bound for $E[\Delta \Phi ]$ by \cref{eq8}. One way to handle this problem is ``packing'': 
Each $C_j$ has a set of $C_k$ which change $C_j$ from unblocked to blocked (or from blocked to unblocked), and we pack such $C_k$ into a set.
This idea will give a ``holistic'' property which tends to give a sharper bound. In the following analysis, we will formalize this idea.
For a blocked circuit ${C_j} \in \Gamma({C_i})$, we define the set
\begin{align*}
	\Gamma '\left( {{C_j}} \right) = \left\{ {{C_k}:{C_k} \in \Gamma \left( {{C_j}} \right)\backslash \left\{ {{C_i}} \right\},{C_j} \in {\alpha _{{C_k}}}} \right\},
\end{align*}
and
\begin{align*}
	E\left[ {{\Delta ^{ * {C_j}}}\Phi } \right] = E\left[ {{\Delta ^{{C_j}}}\Phi } \right] + \sum\limits_{{C_k} \in \Gamma '({C_j})} {\frac{1}{{\left| {{\alpha _{{C_k}}}} \right|}}E\left[ {{\Delta ^{{C_k}}}\Phi } \right]}. 
\end{align*}
For an unblocked circuit ${C_j}\in \Gamma({C_i})$, we define 
\begin{align*}
	\Gamma '\left( {{C_j}} \right) = \left\{ {{C_k}:{C_k} \in \Gamma \left( {{C_j}} \right)\backslash \left\{ {{C_i}} \right\},{C_j} \in {\beta _{{C_k}}}} \right\}
\end{align*}
and
\begin{align*}
	E\left[ {{\Delta ^{ * {C_j}}}\Phi } \right] = E\left[ {{\Delta ^{{C_j}}}\Phi } \right] + \sum\limits_{{C_k} \in \Gamma '({C_j})} {\frac{1}{{\left| {{\beta _{{C_k}}}} \right|}}E\left[ {{\Delta ^{{C_k}}}\Phi } \right]}. 
\end{align*}
We can rewrite \cref{eq8} as 
\begin{align*}
	E\left[ {\Delta \Phi } \right] = \frac{1}{n}\left[ {E\left[ {{\Delta ^{{C_i}}}\Phi } \right] + \sum\limits_{{C_j} \in \Gamma ({C_i})} {E\left[ {{\Delta ^{ * {C_j}}}\Phi } \right]} } \right].
\end{align*}
Firstly, we bound $E[{\Delta ^{ * {C_j}}}\Phi ]$. 
\begin{itemize}
	\item  If ${C_j}$ is blocked.\par
	By \cref{ECj}, we know that $E[{\Delta ^{{C_j}}}\Phi ] = 0$. 
	According to the definition of $\alpha _{C_k}$, if $C_j \in \alpha _{C_k}$, then there exists ${C_k}$ such that ${C_k} \in \left( {\Gamma \left( {{C_j}} \right)} \right)$ and ${\sigma \left( {{C_k}} \right) = 1}$, which implies that $\left| {\Gamma '({C_j})} \right| \le 1$. Thus we have
	\begin{align*}
			E\left[ {{\Delta ^{ * {C_j}}}\Phi } \right] 
			&= E\left[ {{\Delta ^{{C_j}}}\Phi } \right] + \sum\limits_{{C_k} \in \Gamma '({C_j})}  {\frac{1}{{\left| {{\alpha _{{C_k}}}} \right|}}E\left[ {{\Delta ^{{C_k}}}\Phi } \right]} \\
			&= 0 + \sum\limits_{{C_k} \in \Gamma '({C_j})} {\frac{1}{{\left| {{\alpha _{{C_k}}}}  \right|}}\frac{{\left| {{\alpha _{{C_k}}}} \right|}}{{1 + {b^{{\delta _{{C_k}}}}}}}w} \\
			&=\sum\limits_{{C_k} \in \Gamma '({C_j})}\frac{w}{1 + b^{\delta _{C_k}}} \\
			&\le \sum\limits_{{C_k} \in \Gamma '({C_j})}\frac{w}{{1 + {b^\delta }}} \\
			&=\left| {\Gamma '\left( {{C_j}} \right)} \right|\frac{w}{{1 + {b^\delta }}} \\
			&\le \frac{w}{{1 + {b^\delta }}}. 
	\end{align*}
	\item If ${C_j}$ is unblocked.\par
	By \cref{ECj}, we have $E\left[ {{\Delta ^{{C_j}}}\Phi } \right]\; \le \frac{{{b^{{\delta _{{C_j}}}}}}}{{1 + {b^{{\delta _{{C_j}}}}}}}\left( {{\delta _{{C_j}}} - w\left| {B\left( {\sigma ,{C_j}} \right)} \right|} \right)$. 
	Moreover, we have
	\begin{align*}
			E\left[ {{\Delta ^{ * {C_j}}}\Phi } \right] 
			&= E\left[ {{\Delta ^{{C_j}}}\Phi } \right] + \sum\limits_{{C_k} \in \Gamma '({C_j})} {\frac{1}{{\left| {{\beta _{{C_k}}}} \right|}}E\left[ {{\Delta ^{{C_k}}}\Phi } \right]} \\
			&\le \frac{{{b^{{\delta _{{C_j}}}}}}}{{1 + {b^{{\delta _{{C_j}}}}}}}\left( {{\delta _{{C_j}}} - w\left| {B\left( {\sigma ,{C_j}} \right)} \right|} \right) - \sum\limits_{{C_k} \in \Gamma '({C_j}),\sigma \left( {\Gamma \left( {{C_k}} \right)} \right) = 0} {\frac{1}{{\left| {{\beta _{{C_k}}}} \right|}}\frac{{{b^{{\delta _{{C_k}}}}}\left| {{\beta _{{C_k}}}} \right|}}{{1 + {b^{{\delta _{{C_k}}}}}}}w} \\
			&\le \frac{{{b^{{\delta _{{C_j}}}}}}}{{1 + {b^{{\delta _{{C_j}}}}}}}\left( {{\delta _{{C_j}}} - w\left| {B(\sigma ,{C_j})} \right|} \right) \\
			&\le \frac{{b\delta }}{{1 + {b^\delta }}}.
	\end{align*}
\end{itemize}
Note that $w = \frac{{{b\delta }}}{{{b^\delta } + 2}}$. Thus $E[\Delta \Phi ]$ can be bounded as follows.
\begin{equation}\label{E}
\begin{aligned}
	nE\left[ {\Delta \Phi } \right] 
	&= \left[ {E\left[ {{\Delta ^{{C_i}}}\Phi } \right] + \sum\limits_{{C_j} \in \Gamma ({C_i})} {E\left[ {{\Delta ^{ * {C_j}}}\Phi } \right]} } \right] \\
	&= \left[ { - {\delta _{{C_i}}} + w\left| {B\left( {\sigma ,{C_i}} \right)} \right| + \sum\limits_{{C_j} \in \Gamma \left( {{C_i}} \right),{C_j} \in B\left( {\sigma ,{C_i}} \right)} {E\left[ {{\Delta ^{ * {C_j}}}\Phi } \right]}  + \sum\limits_{{C_j} \in \Gamma \left( {{C_i}} \right),{C_j} \notin B\left( {\sigma ,{C_i}} \right)} {E\left[ {{\Delta ^{ * {C_j}}}\Phi } \right]} } \right] \\
	&\le \left[ { - {\delta _{{C_i}}} + \sum\limits_{{C_j} \in \Gamma \left( {{C_i}} \right),{C_j} \in B\left( {\sigma ,{C_i}} \right)} {\frac{{2 + {b^\delta }}}{{1 + {b^\delta }}}w}  + \sum\limits_{{C_j} \in \Gamma \left( {{C_i}} \right),{C_j} \notin B\left( {\sigma ,{C_i}} \right)} {\frac{{b\delta }}{{1 + {b^\delta }}}} } \right] \\
	&= \frac{1}{\left( {1 + {b^\delta }} \right)}\left[ -(1 + {b^\delta })\delta _{{C_i}}+\left|B\left( {\sigma ,{C_i}} \right) \right|(2 + {b^\delta })w+(\delta_{C_i}-\left|B\left( {\sigma ,{C_i}} \right) \right|)b\delta\right]\\
	&= \frac{{{\delta _{{C_i}}}}}{{\left( {1 + {b^\delta }} \right)}}\left( {b\delta  - 1 - {b^\delta }} \right).
\end{aligned}
\end{equation}
Notice that $E[\Delta \Phi ] \le 0 $, when $b \le \frac{1}{\delta }$.

Now we are ready to bound $E\left[ {\Delta \Phi } \right]$, which will bound the mixing time. We define ${\beta _{\sigma ,{\sigma _{{C_i}}}}}$ such that
\begin{align*}
	E\left[ {\Phi \left( {\sigma ',{{\sigma '}_{{C_i}}}} \right)} \right] = {\beta _{\sigma ,{\sigma _{{C_i}}}}}\Phi \left( {\sigma ,{\sigma _{{C_i}}}} \right).
\end{align*}
Thus we have 
\begin{align*}
	E\left[ {\Delta \Phi } \right] = E\left[ {\Phi \left( {\sigma ',{{\sigma '}_{{C_i}}}} \right)} \right] - \Phi \left( {\sigma ,{\sigma _{{C_i}}}} \right) = \left( {{\beta _{\sigma ,{\sigma _{{C_i}}}}} - 1} \right)\Phi \left( {\sigma ,{\sigma _{{C_i}}}} \right),
\end{align*}
and
\begin{align*}
	{\beta _{\sigma ,{\sigma _{{C_i}}}}} = 1{\rm{ + }}\frac{{E\left[ {\Delta \Phi } \right]}}{{\Phi \left( {\sigma ,{\sigma _{{C_i}}}} \right)}}.
\end{align*}
According to Lemma \ref{coupling}, we set $\beta  = {\max _{\sigma ,{\sigma _{{C_i}}}}}{\beta _{\sigma ,{\sigma _{{C_i}}}}}$. From \cref{potential}, i.e., the definition of $\Phi $, it is easy to know that $ {\Phi \left( {\sigma ,{\sigma _{{C_i}}}} \right)} \le {\delta _{{C_i}}} $. By the bound of $E[\Delta\Phi]$ in \cref{E}, we have
\begin{align*}
	\beta  \le 1{\rm{ + }}\frac{1}{{n\left( {1 + {b^\delta }} \right)}}\left( {b\delta  - 1 - {b^\delta }} \right).
\end{align*}
For any states $\sigma $ and $\eta $, we have $\Phi (\sigma ,\eta ) \le n\delta $. 
Thus, we have
\begin{equation}\label{varepsilon}
	\begin{aligned}
		E\left[ {\Phi \left( {{X_t},{Y_t}} \right)} \right] 
		&\le {\beta ^t} \cdot E\left[ {\Phi \left( {{X_0},{Y_0}} \right)} \right] \\
		&\le {\left( {1{\rm{ + }}\frac{1}{{n\left( {1 + {b^\delta }} \right)}}\left( {b\delta  - 1 - {b^\delta }} \right)} \right)^t} \cdot n\delta \\
		&\le {e^{ \left( {\frac{1}{{n\left( {1 + {b^\delta }} \right)}}\left( {b\delta  - 1 - {b^\delta }} \right)} \right)t}} \cdot n\delta. 
	\end{aligned}
\end{equation}
\cref{varepsilon} and \cref{coupling} produce a bound of the mixing time
\begin{align*}
	\tau (\varepsilon ) \le \frac{{n\left( {1 + {b^\delta }} \right)}}{{{b^\delta } + 1 - b\delta }}\log \left( {\frac{{n\delta }}{\varepsilon }} \right).
\end{align*}
\begin{theorem}
	(Mixing time of Markov chain). For the six-vertex model with the constraint function $f^*$, whose underlying graph is two-by-two-intersection free, the Markov Chain $\{Z_t\}$ mixes in time $O\left( {\frac{{n\left( {1{\rm{ + }}{b^\delta }} \right)}}{{{b^\delta } + 1 - b\delta }}\log \left( {n\delta /\varepsilon } \right)} \right)$ for $\delta  \ge 2$, $b \le \frac{1}{\delta }$. 
\end{theorem}


\bibliography{bibref}
\appendix

\end{document}